\documentclass[dvipdfmx]{amsart}

\usepackage{amssymb}
\usepackage{amscd}
\usepackage{epsfig}
\usepackage{color}
\usepackage{bm}

\usepackage{mathtools}

\usepackage[setpagesize=false]{hyperref}

\usepackage{ascmac}

\usepackage[all]{xy}
\usepackage{time}

\usepackage[OT2,T1]{fontenc}
\DeclareSymbolFont{cyrletters}{OT2}{wncyr}{m}{n}
\DeclareMathSymbol{\Sha}{\mathalpha}{cyrletters}{"58}

\usepackage{stackengine}
\stackMath
\newcommand\undertilde[2][1]{%
 \def\useanchorwidth{T}%
  \ifnum#1>1%
    \stackunder[0pt]{\tenq[\numexpr#1-1\relax]{#2}}{\scriptscriptstyle\sim}%
  \else%
    \stackunder[1pt]{#2}{\scriptscriptstyle\sim}%
  \fi%
}

\usepackage{yhmath}
\usepackage {wasysym}

\DeclareMathAlphabet{\mathpzc}{OT1}{pzc}{m}{it}

\usepackage{ulem}

\usepackage{tikz}

\title[Notes on Ecalle's and Brown's solutions]
{Notes on Ecalle's and Brown's solutions to the double shuffle relations modulo products}


\author{Hidekazu Furusho}
\author{Minoru Hirose}
\author{Nao Komiyama}

\address{Graduate School of Mathematics, Nagoya University,
Furo-cho, Chikusa-ku, Nagoya, 464-8602, Japan}
\email{furusho@math.nagoya-u.ac.jp}
\address{Faculty of Science
Kagoshima University, 1-21-35 Korimoto, Kagoshima,
890-0065, Japan}
\email{hirose@sci.kagoshima-u.ac.jp}
\address{Department of Mathematics, Graduate School of Science, Osaka University Toyonaka, Osaka 560-0043, Japan}
\email{komiyama.nao.aww@osaka-u.ac.jp}

\date{January 11, 2026}
\dedicatory{Dedicated to Professor Hiroaki Nakamura \\
on the occasion of his 60th birthday}

\newtheorem{thm}{Theorem}%[section]
\newtheorem{lem}[thm]{Lemma}

\newtheorem{prop}[thm]{Proposition}

{\theoremstyle{definition} \newtheorem{rem}[thm]{Remark}}
{\theoremstyle{definition} \newtheorem{defn}[thm]{Definition}}
{\theoremstyle{definition} \newtheorem{thm-defn}[thm]{Theorem-Definition}}
{\theoremstyle{definition}  \newtheorem{eg}[thm]{Example}}
{\theoremstyle{definition}  }

{\theoremstyle{remark} }

\numberwithin{equation}{section}
\newcommand{\Q}{\mathbb{Q}}

\newcommand{\Z}{\mathbb{Z}}
\newcommand{\N}{\mathbb{N}}

\newcommand{\shuffle}{\scalebox{.8}{$\Sha$}}

\newcommand{\vecx}{{\bf x}}

\newcommand{\nega}{\mathsf{neg}}
\newcommand{\swap}{\mathsf{swap}}

\newcommand{\ARI}{\mathsf{ARI}}
\newcommand{\BIMU}{{\rm BIMU}}

\newcommand{\ulflex}[2]{{#1}\rceil_{\scalebox{.7}{$#2$}}}
\newcommand{\urflex}[2]{{}_{\scalebox{.7}{$#1$}}\lceil{#2}}

\newcommand{\arit}{\mathsf{arit}}
\newcommand{\ari}{\mathsf{ari}}
\newcommand{\lu}{\mathsf{lu}}

\newcommand{\preari}{\mathsf{preari}}
\newcommand{\expari}{\mathsf{expari}}
\newcommand{\logari}{\mathsf{logari}}
\newcommand{\garit}{\mathsf{garit}}
\newcommand{\gari}{\mathsf{gari}}
\newcommand{\ganit}{{\sf ganit}}
\newcommand{\invgari}{{\sf invgari}}
\newcommand{\pal}{{\sf pal}}
\newcommand{\dupal}{{\sf dupal}}
\newcommand{\dur}{{\sf dur}}
\newcommand{\pic}{{\sf pic}}
\newcommand{\adari}{{\sf adari}}
\newcommand{\shmap}{\mathpzc{Sh}}
\newcommand{\unitmould}{1_{{\mathcal M}(\mathcal{F})}}

\newcommand{\unitdimould}{1_{\mathcal{M}_2(\mathcal F)}}

\newcommand{\pol}{\mathsf{pol}}
\newcommand{\al}{\mathsf{al}}
\newcommand{\il}{\mathsf{il}}

\newcommand{\id}{\mathsf{id}}

\newcommand{\GARI}{\mathsf{GARI}}
\newcommand{\as}{\mathsf{as}}

\newcommand{\paj}{\mathsf{paj}}

\newcommand{\Lau}{\mathsf{Lau}}
\newcommand{\ser}{\mathsf{ser}}

\newcommand{\sang}{\mathsf{sang}}
\newcommand{\slang}{\mathsf{slang}}

\newcommand{\sa}{\mathsf{sa}}
\newcommand{\mupaj}{\mathsf{mupaj}}

\newcommand{\luma}{\mathsf{luma}}

\newcommand{\leng}{\mathsf{leng}}

\begin{document}
%%\texttt{\jobname.tex}
%\hfill
%\texttt{\jobname.tex}
%\date{ \now, \today}
\bibliographystyle{amsalpha+}
\maketitle

%%%%%%%%%%%%%%%%%%%%%%%%%%%%%%%%%%%%%%%%%%%%%%%%%%%%%%%%%%%%%%%%%%%%%%
\begin{abstract}
We investigate relationships between polar/polynomial solutions to the double shuffle relations modulo products,
which were independently introduced by Brown and Ecalle.
\end{abstract}

%%%%%%%%%%%%%%%%%%%%%%%%%%%%%%%%%%%%%%%%%%%%%%%%%%%%%%%%%%%%%%%%%%%%%%
%\tableofcontents
{\small \tableofcontents}

%%%%%%%%%%%%%%%%%%%%%%%%%%%%%%%%%%%%%%%%%%%%%%%%%%%%%%%%%%%%%%%%%%%%%%
\setcounter{section}{-1}
\section{Introduction}\label{introduction}
Multiple zeta values (MZVs) satisfy two fundamental families of algebraic relations,
the shuffle and stuffle (or harmonic) relations.
Taken together, these give rise to the so-called double shuffle relations,
which play a central role in the algebraic study of MZVs. % and their motivic counterparts.
Understanding the structure of the space of solutions to these relations, modulo products, has been a topic of interest in the study of MZV's.

In this context, Brown introduced
explicit polar and polynomial solutions $\psi_{2n+1}$ ($n\geq 1$) and $\psi_{-1}$,
%in his work on the anatomy of associators
in \cite{B-anatomy}, as well as the polynomial
solutions  $\sigma^{c}_{2n+1}$ %in his study of zeta element
in  \cite{B17b},
to the double shuffle relations modulo products, built in the framework of motivic MZVs.
Prior to Brown's work,  Ecalle \cite{E-flex}
had already developed a sophisticated  construction
of  polynomial solution  $\luma_{2n+1}$ within his theory of moulds.
Although both frameworks encode closely related algebraic structures, their precise correspondence has not yet been completely understood.
In this paper, we provide a reinterpretation of Brown's polar solutions
$\psi_{2n+1}$ and $\psi_{-1}$ in terms of Ecalle's mould theory in \S \ref{sec: On polar solutions} (Theorems \ref{thm 2n+1} and \ref{thm -1}).
In \S \ref{sec: On polynomial solutions} we give  a mould-theoretic interpretation of Brown's solution polynomial $\sigma^{c}_{2n+1}$ and compare it with $\luma_{2n+1}$
up to length (depth) 3 (Theorem \ref{thm: comparison depth 3}).

We remind that
certain maps for constructing solutions to the double shuffle relations modulo products have been studied independently
by Ecalle \cite{E-flex} and Brown \cite{B-anatomy}.
A comparison of these two maps is provided in \cite{MT}.

%%%%%%%%%%%%%%%%%%%%%%%%%%%%%%%%%%%%%%%%%%%%%%%%%%%%%%%%%%%%%%%%%%%%%%
\section{Preparation}\label{Preparation}
This section reviews the minimal background on mould theory needed to formulate and explain the main results of this paper.
To help the reader become familiar with the techniques of mould theory
and to facilitate later explicit computations, examples of several basic notions are presented up to depth two or three.

%%%%%%%%%%%%%%%%%%%%%%%%%%%%%%%%
%\subsection{Algebraic framework}

%%%%%%%%%%%%%%%%%%%%%%%%%%%%%%%%
\subsection{Basic mould operations}
%{The operations $\ari$, $\gari$ and $\adari$}
In this subsection, basic mould operations, $\ari$, $\gari$ and $\adari$,
%introduced in \cite{E-flex}
are recalled (consult \cite{E-flex} and \cite{S-ARIGARI} for details). %, together with their precise definitions.

The notion of moulds are introduced by Ecalle in \cite{E81},
but here a slightly different formulation will be used.
Let $\mathcal F:=\cup_m \mathcal F_m$ be a family of functions (cf. \cite{FHK}).
In this paper, it suffices to restrict to the following three cases: 
the polynomial case
$\mathcal F_\pol$ with ${\mathcal F}_{\pol,m}=\Q[x_1,..,x_m]$,
the series case
$\mathcal F_\ser$ with ${\mathcal F}_{\ser,m}=\Q[[x_1,..,x_m]]$
and the Laurent-series case
$\mathcal F_\Lau$ with
${\mathcal F}_{\Lau,m}=\Q((x_1,..,x_m))$ which is the fraction filed of
$\Q[[x_1,..,x_m]]$.

\begin{defn}%[\cite{E81} I, pp.12-13 and \cite{S-ARIGARI}]
A {\it mould} $M$ with values in  $\mathcal F=\cup_m \mathcal F_m$
means a collection
$M=(M^m(x_1,..,x_m))_{m\geq 0}$ with
$M^m(x_1,..,x_m)\in\mathcal F_m$ for each $m\geq 0$
(cf. \cite{S-ARIGARI}).
The set $\mathcal M(\mathcal F)$ of moulds with values in $\mathcal F$
forms a $\Q$-algebra under the product
$$
M\times N=\left(\sum_{k=0}^m M^k(x_1,..,x_k)N^{m-k}(x_{k+1},..,x_m)
\right)_{m\ge0}
$$
for $M=(M^m(x_1,..,x_m))_{m\geq 0}$ and $N=(N^m(x_1,..,x_m))_{m\geq 0}$.
Denote the unit of $(\mathcal M(\mathcal F), \times )$ as $1_{\mathcal M(\mathcal F)}$ defined by
\begin{equation*}
\unitmould^{m}(x_1, \dots, x_m)
:=\left\{
\begin{array}{ll}
	1 & (m=0), \\
	0 & (\mbox{otherwise}).
\end{array}\right.
\end{equation*}
\end{defn}

The subset $\ARI(\mathcal F)$ which consists of moulds $M$ with $M^0(\emptyset)=0$ forms a subalgebra of $\mathcal M(\mathcal F)$ and
also forms a Lie algebra under the bracket
$\lu(M,N)=M\times N- N\times M$.
The subset $\GARI(\mathcal F)$ of $\mathcal M(\mathcal F)$ which consists of moulds $M$ with $M^0(\emptyset)=1$
forms a group under the product.

To explain that
$\ARI(\mathcal F)$ (resp. $\GARI(\mathcal F)$) carries another Lie bracket $\ari$ (resp. a product $\gari$),
we recall Ecalle's notion of flexions (cf. \cite{E-flex}).
Here we employ the reformulation given in \cite[Definition 1.14]{FK}:
Put $X:=\{ x_i \}_{i\ge1}$.
Let $X_\Z$ be the free $\Z$ module generated by $X$, that is, $X_\Z$ is defined by
$$
X_\Z:=\{ a_1x_1+\cdots+a_kx_k \ |\ k\ge1,\ a_j\in\Z \},
$$
and let $X_\Z^\bullet$ be the non-commutative free monoid generated by all elements of $X_\Z$ with the empty word $\emptyset$ as the unit.
For $\omega=(u_1,\dots,u_r)\in X_\Z^\bullet$ with $u_1,\dots,u_r\in X_\Z$, We call $r$ {\it the length of $\omega$} and denote it by $l(\omega)$.
We put  $\vecx_m:=(x_1,\dots,x_m)\in X_\Z^\bullet$.

\begin{defn}%[{\cite[Definition 1.14]{FK}}]
\label{def:flexion}
The {\it flexions} are two binary operators $\urflex{*}{*},\ \ulflex{*}{*}:X_{\Z}^\bullet\times X_{\Z}^\bullet\rightarrow X_{\Z}^\bullet$
 which are defined by
\begin{align*}
	\urflex{\beta}{\alpha}
	&:= (b_1+\cdots+b_n+a_1, a_2, \dots, a_m), \\
	\ulflex{\alpha}{\beta}
	&:=(a_1, \dots, a_{m-1}, a_m+b_1+\cdots+b_n), \\
	\urflex{\emptyset}{\gamma}&:=\ulflex{\gamma}{\emptyset}:=\gamma ,
	\qquad \urflex{\gamma}{\emptyset}:=\ulflex{\emptyset}{\gamma}:=\emptyset ,
\end{align*}
for $\alpha=(a_1,\dots,a_m)$, $\beta=(b_1,\dots,b_n)\in X_{\Z}^\bullet$ ($m,n\geq1$) and $\gamma\in X_{\Z}^\bullet$.
\end{defn}

The following map $\arit$ is required to endow $\ARI(\mathcal F)$  a pre-Lie algebra  structure.

\begin{defn}
For $N\in {\mathcal M}(\mathcal F)$,
the map
$$
\arit(N): \mathcal M(\mathcal F)\to \mathcal M(\mathcal F)
$$
is defined as follows:
for $N\in {\mathcal M}(\mathcal F)$,
\begin{align*}
&(\arit(N)(M))^0(\vecx_0):=(\arit(N)(M))^1(\vecx_1):=0,
\intertext{and for $m\ge 2$,}
&(\arit(N)(M))^m(\vecx_m)
:=\sum_{\substack{\vecx_m=\alpha\beta\gamma \\ \beta,\gamma\neq\emptyset}}
M(\alpha\urflex{\beta}{\gamma})N(\beta)
-\sum_{\substack{\vecx_m=\alpha\beta\gamma \\ \alpha,\beta\neq\emptyset}}
M(\ulflex{\alpha}{\beta}\gamma)N(\beta).
\end{align*}
\end{defn}

\begin{eg}
For $M,N\in {\mathcal M}(\mathcal F)$, we have
\begin{align*}
&(\arit(N)(M))^0(\emptyset)=(\arit(N)(M))^1(x_1)=0, \\
&(\arit(N)(M))^2(x_1,x_2)=M^1(x_1+x_2)\left\{ N^1(x_1) - N^1(x_2) \right\}, \\
&(\arit(N)(M))^3(x_1,x_2,x_3)
=M^2(x_1,x_2+x_3)\left\{ N^1(x_2) - N^1(x_3) \right\} \\
&\quad+ M^2(x_1+x_2,x_3)\left\{ N^1(x_1) - N^1(x_2) \right\}
+ M^1(x_1+x_2+x_3)\left\{ N^2(x_1,x_2) - N^2(x_2,x_3) \right\}.
\end{align*}
\end{eg}

\begin{defn}
The map
$$
\preari:\mathcal M(\mathcal F)\times \mathcal M(\mathcal F)\to \mathcal M(\mathcal F)
$$
is defined as follows:
For $M,N\in {\mathcal M}(\mathcal F)$,
$$
\preari(M,N):=\arit(N)(M)+M\times N,
$$
and for $n\geq0$ and $A\in \ARI(\mathcal F)$,
$$
\preari_n(A):=\left\{\begin{array}{ll}
\unitmould& (n=0), \\
A& (n=1), \\
\preari(\preari_{n-1}(A),A)&  (n\geq2).
\end{array}
\right.
$$
\end{defn}

\begin{eg}
For $A\in \ARI(\mathcal F)$, we have
\begin{align*}
&(\preari_2(A))^0(\emptyset)=(\preari_2(A))^1(x_1)=0, \\
&(\preari_2(A))^2(x_1,x_2)=A^1(x_1+x_2)\left\{ A^1(x_1) - A^1(x_2) \right\} + A^1(x_1)A^1(x_2), \\
&(\preari_2(A))^3(x_1,x_2,x_3)
=A^2(x_1,x_2+x_3)\left\{ A^1(x_2) - A^1(x_3) \right\} \\
&\quad+ A^2(x_1+x_2,x_3)\left\{ A^1(x_1) - A^1(x_2) \right\}
%\\ &\qquad
+ A^1(x_1+x_2+x_3)\left\{ A^2(x_1,x_2) - A^2(x_2,x_3) \right\} \\
&\qquad\quad +A^1(x_1)A^2(x_2,x_3) +A^2(x_1,x_2)A^1(x_3),
\end{align*}
and have
\begin{align*}
&(\preari_3(A))^0(\emptyset)=(\preari_3(A))^1(x_1)=(\preari_3(A))^2(x_1,x_2)=0, \\
&(\preari_3(A))^3(x_1,x_2,x_3) \\
&=\left[ A^1(x_1+x_2+x_3)\left\{ A^1(x_1) - A^1(x_2+x_3) \right\} + A^1(x_1)A^1(x_2+x_3) \right]\left\{ A^1(x_2)-A^1(x_3) \right\}
\\&\quad
+ \left[ A^1(x_1+x_2+x_3)\left\{ A^1(x_1+x_2) - A^1(x_3) \right\} + A^1(x_1+x_2)A^1(x_3) \right]\left\{ A^1(x_1)-A^1(x_2) \right\} \\
&\qquad + \left[ A^1(x_1+x_2)\left\{ A^1(x_1) - A^1(x_2) \right\} + A^1(x_1)A^1(x_2) \right]A^1(x_3).
\end{align*}
\end{eg}

It is shown in \cite[Proposition A.6]{FK} with a complete proof
that $\ARI(\mathcal F)$ forms a pre-Lie algebra under  $\preari$,
and whence $\ARI(\mathcal F)$ forms a Lie algebra under the following  $\ari$-bracket.

\begin{defn}
The map
$$
\ari:\ARI(\mathcal F)\times \ARI(\mathcal F)\to \ARI(\mathcal F)
$$
is defined by
$$
\ari(M,N)=\preari(M,N)-\preari(N,M)
$$
for $M,N\in {\mathcal M}(\mathcal F)$.
\end{defn}

\begin{eg}
%\Add{Give  examples of ari-bracket  up to depth 3}
For $M,N\in {\mathcal M}(\mathcal F)$, we have
\begin{align*}
&\ari(M,N)^0(\emptyset)=\ari(M,N)^1(\emptyset)(x_1)=0, \\
&\ari(M,N)^2(x_1,x_2)=M^1(x_1+x_2)\left\{ N^1(x_1) - N^1(x_2) \right\}
-N^1(x_1+x_2)\left\{ M^1(x_1) - M^1(x_2) \right\}, \\
&\ari(M,N)^3(x_1,x_2,x_3)
=M^2(x_1,x_2+x_3)\left\{ N^1(x_2) - N^1(x_3) \right\} \\
&
-N^2(x_1,x_2+x_3)\left\{ M^1(x_2) - M^1(x_3) \right\}
\\
&+ M^2(x_1+x_2,x_3)\left\{ N^1(x_1) - N^1(x_2) \right\}
-N^2(x_1+x_2,x_3)\left\{ M^1(x_1) - M^1(x_2) \right\}
\\
&+ M^1(x_1+x_2+x_3)\left\{ N^2(x_1,x_2) - N^2(x_2,x_3) \right\}
-N^1(x_1+x_2+x_3)\left\{ M^2(x_1,x_2) - M^2(x_2,x_3) \right\}.
\end{align*}
\end{eg}

Next we equip $\GARI(\mathcal F)$ with an operation $\gari$.

\begin{defn}
For $T\in \GARI(\mathcal F)$,
the map
$$
\garit(T):\GARI(\mathcal F)\to\GARI(\mathcal F)
$$
is defined by,  for  $S\in \GARI(\mathcal F)$,
\begin{align*}
&(\garit(T)(S))^0(\vecx_0):=1
\intertext{and for $m\ge 1$,}
&(\garit(T)(S))^m(\vecx_m) \\
&:=\sum_{s\geq1}\sum_{\substack{\vecx_m=a_1b_1c_1\cdots a_sb_sc_s \\ b_i,\,c_ia_{i+1}\neq\emptyset}}
S(\urflex{a_1}{\ulflex{b_1}{c_1}}\cdots \urflex{a_s}{\ulflex{b_s}{c_s}})
T(a_1)\cdots T(a_s) T^{\times-1}(c_1)\cdots T^{\times-1}(c_s),
\end{align*}
where $T^{\times-1}$ is the inverse element of $T$ in the group $(\GARI(\mathcal F),\times)$.
\end{defn}
\begin{eg}
For $S,T\in \GARI(\mathcal F)$, we have
\begin{align*}
&(\garit(T)(S))^1(x_1)=S^1(x_1), \\
&(\garit(T)(S))^2(x_1,x_2)=S^2(x_1,x_2) + S^1(x_1+x_2)\left\{ T^1(x_1) - T^1(x_2) \right\}, \\
&(\garit(T)(S))^3(x_1,x_2,x_3)
=S^3(x_1,x_2,x_3) \\
&\quad + S^2(x_1,x_2+x_3)\left\{ T^1(x_2) - T^1(x_3) \right\} + S^2(x_1+x_2,x_3)\left\{ T^1(x_1) - T^1(x_2) \right\} \\
&\qquad + S^1(x_1+x_2+x_3)\left\{ T^2(x_1,x_2) - T^2(x_2,x_3) - T^1(x_1)T^1(x_3) \right\}.
\end{align*}
\end{eg}

{It is shown  in \cite[Proposition 39]{FHK} that $\GARI(\mathcal F)$ forms a group under the following $\gari$-product. }
%However no proofs are provided in any references.
%Give a proof by using mould proper theorem.}

\begin{defn}
The map
$$
\gari:\GARI(\mathcal F)\times \GARI(\mathcal F)\to \GARI(\mathcal F)
$$
is defined by,
for $S,T\in \GARI(\mathcal F)$,
$$
\gari(S,T):=\garit(T)(S)\times T.
$$
\end{defn}
\begin{eg}
For $S,T\in \GARI(\mathcal F)$, we have
\begin{align*}
&(\gari(S,T))^0(x_1)=1, \\
&(\gari(S,T))^1(x_1)=S^1(x_1) + T^1(x_1), \\
&(\gari(S,T))^2(x_1,x_2)=S^2(x_1,x_2) + S^1(x_1+x_2)\left\{ T^1(x_1) - T^1(x_2) \right\} \\
&\hspace{3.3cm}+ S^1(x_1)T^1(x_2) + T^2(x_1,x_2), \\
&(\gari(S,T))^3(x_1,x_2,x_3)
=S^3(x_1,x_2,x_3) \\
&\quad+ S^2(x_1,x_2+x_3)\left\{ T^1(x_2) - T^1(x_3) \right\} + S^2(x_1+x_2,x_3)\left\{ T^1(x_1) - T^1(x_2) \right\} \\
&\qquad + S^1(x_1+x_2+x_3)\left\{ T^2(x_1,x_2) - T^2(x_2,x_3) + T^1(x_2)T^1(x_3) - T^1(x_1)T^1(x_3) \right\} \\
&\qquad\quad + \left[ S^2(x_1,x_2) + S^1(x_1+x_2)\left\{ T^1(x_1) - T^1(x_2) \right\} \right]T^1(x_3) \\
&\qquad\qquad + S^1(x_1)T^2(x_2,x_3) + T^3(x_1,x_2,x_3).
\end{align*}
\end{eg}

%\begin{rem}
%Denote the inverse element of $S$ in the group $(\GARI(\mathcal F),\gari)$ by
%$$
%\invgari(S).
%$$
%\end{rem}

The Lie algebra $\ARI(\mathcal F)$ and the group $\GARI(\mathcal F)$
are related  via the following exponential map.

\begin{defn}
The map
$$\expari:\ARI(\mathcal F) \rightarrow \GARI(\mathcal F)$$
is defined by,
for $A\in \ARI(\mathcal F)$,
\begin{align*}
\expari(A)
&:=\sum_{n\geq0}\frac{1}{n!}\preari_n(A) \\
&=\unitmould+A+\frac{1}{2}\preari_2(A)+\frac{1}{6}\preari_3(A) +\cdots, \\
&=\unitmould+A+\frac{1}{2}\preari(A,A)+\frac{1}{6}\preari(\preari(A,A),A) +\cdots .
\end{align*}
\end{defn}
\begin{eg}
For $A\in \ARI(\mathcal F)$, we have
\begin{align*}
&(\expari(A))^0(\emptyset)=1, \\ % \qquad
&(\expari(A))^1(x_1)=A^1(x_1), \\
&(\expari(A))^2(x_1,x_2)
=A^2(x_1,x_2) +\frac{1}{2}\left[ A^1(x_1+x_2)\left\{ A^1(x_1) - A^1(x_2) \right\} + A^1(x_1)A^1(x_2) \right], \\
&(\expari(A))^3(x_1,x_2,x_3)
=A^3(x_1,x_2,x_3) \\
&\quad+\frac{1}{2}\left[A^2(x_1,x_2+x_3)\left\{ A^1(x_2) - A^1(x_3) \right\} + A^2(x_1+x_2,x_3)\left\{ A^1(x_1) - A^1(x_2) \right\} \right. \\
&\quad \left. + A^1(x_1+x_2+x_3)\left\{ A^2(x_1,x_2) - A^2(x_2,x_3) \right\} +A^1(x_1)A^2(x_2,x_3) +A^2(x_1,x_2)A^1(x_3) \right] \\
&\qquad +\frac{1}{6}\Bigl(
\left[ A^1(x_1+x_2+x_3)\left\{ A^1(x_1) - A^1(x_2+x_3) \right\} + A^1(x_1)A^1(x_2+x_3) \right]\left\{ A^1(x_2)-A^1(x_3) \right\} \\
&\qquad+ \left[ A^1(x_1+x_2+x_3)\left\{ A^1(x_1+x_2) - A^1(x_3) \right\} + A^1(x_1+x_2)A^1(x_3) \right]\left\{ A^1(x_1)-A^1(x_2) \right\} \\
&\qquad + \left[ A^1(x_1+x_2)\left\{ A^1(x_1) - A^1(x_2) \right\} + A^1(x_1)A^1(x_2) \right]A^1(x_3)
\Bigr).
\end{align*}
\end{eg}

\begin{rem}
As the above example suggests, we prove inductively that
the map $\expari:\ARI(\mathcal F) \rightarrow \GARI(\mathcal F)$ is bijective.
We denote its inverse map from $\GARI(\mathcal F)$ to $\ARI(\mathcal F)$ by $\logari$.
\end{rem}

\begin{defn}
For $S\in \GARI(\mathcal F)$,
the map
$$
\adari(S):\ARI(\mathcal F)\to \ARI(\mathcal F)
$$
is defined by
$$
\adari(S)(A):=\logari( \gari(S,\expari(A),\invgari(S)) )
$$
for $A\in \ARI(\mathcal F)$.
Here we denote  the inverse element of $S$ in the group $(\GARI(\mathcal F),\gari)$ by
$\invgari(S)$.
\end{defn}

\begin{eg}
For $S\in \GARI(\mathcal F)$ and $A\in \ARI(\mathcal F)$, we have
\begin{align*}
&(\adari(S)(A))^0(\emptyset)=0, \\
&(\adari(S)(A))^1(x_1)=A^1(x_1), \\
&(\adari(S)(A))^2(x_1,x_2)
=A^2(x_1,x_2) + \{ S^1(x_1+x_2)-S^1(x_2) \}A^1(x_1) \\
&\quad - \{ S^1(x_1+x_2)-S^1(x_1) \}A^1(x_2)
%&\qquad
+ \{ S^1(x_2)-S^1(x_1) \}A^1(x_1+x_2).
\end{align*}
\end{eg}

%%%%%%%%%%%%%%%%%%%%%%%%%%%%
\subsection{Alternality and symmetrality}
In this subsection we recall Ecalle's notion of alternality and symmetrality
introduced in \cite{E92}
and we present our formulation of the notions in terms of Sauzin's dimoulds (\cite{Sau}).

\begin{defn}[cf. {\cite[Definition 5.2]{Sau}}]\label{def:2.2.1}
A {\it dimould} $M$ with values in  $\mathcal F$ is a collection
\begin{equation*}
	M:=\left( M^{r,s}(x_1, \dots, x_r;\,x_{r+1}, \dots, x_{r+s}) \right)_{r,s\ge0},
\end{equation*}
with $M^{r,s}(x_1, \dots, x_r;\,x_{r+1}, \dots, x_{r+s}) \in \mathcal{F}_{r+s}$ for $r, s\geq 0$.
We denote the set of all dimoulds with values in $\mathcal F$ by $\mathcal{M}_2(\mathcal F)$.
By the component-wise summation and the component-wise scalar multiple, the set $\mathcal{M}_2(\mathcal F)$ forms a $\Q$-linear space.
The {\it product} of $\mathcal{M}_2(\mathcal F)$ which is denoted by the same symbol $\times$ as the product of $\mathcal{M}(\mathcal F)$ is defined by
\begin{align}\label{eqn:2.2.1}
	&(A\times B)^{r,s}(x_1, \dots, x_r;\,x_{r+1}, \dots, x_{r+s}) \\
	&:=\sum_{i=0}^r\sum_{j=0}^s
	A^{r}(x_1, \dots, x_i;\,x_{r+1}, \dots, x_{r+j})
	B^{s}(x_{i+1}, \dots, x_r;\,x_{r+j+1}, \dots, x_{r+s}), \nonumber
\end{align}
for $A,B\in\mathcal{M}_2(\mathcal F)$ and for $r,s\geq0$.
Then $(\mathcal{M}_2(\mathcal F), \times)$ is a non-commutative, associative $\Q$-algebra.
The unit $\unitdimould$ of $(\mathcal{M}_2(\mathcal F), \times)$ is given by
\begin{equation*}
\unitdimould^{r,s}(x_1, \dots, x_r;\,x_{r+1}, \dots, x_{r+s})
:=\left\{
\begin{array}{ll}
	1 & (r=s=0), \\
	0 & (\mbox{otherwise}).
\end{array}\right.
\end{equation*}
\end{defn}

\begin{defn}\label{def:tensor product}
Let
$$i_\otimes:\mathcal{M}(\mathcal F) \otimes_\Q \mathcal{M}(\mathcal F) \rightarrow \mathcal{M}_2(\mathcal F)$$
be the $\Q$-linear map defined by
$$
i_\otimes(M\otimes N)^{r,s}(x_1, \dots, x_r;\,x_{r+1}, \dots, x_{r+s})
:=M^{r}(x_1, \dots, x_r) N^{s}(x_{r+1}, \dots, x_{r+s})
$$
for $r,s\geq0$.
By abuse of notation, we denote $i_\otimes(M\otimes N)$ simply
by $M\otimes N$ and call it as the {\it tensor product} of $M$ and $N$.
\end{defn}

We put
$\mathcal A_X:=\Q \langle X_\Z \rangle$
to be the non-commutative polynomial $\Q$-algebra generated by
$X_\Z$, and we equip $\mathcal A_X$ a product $\shuffle:\mathcal A_X^{\otimes2} \rightarrow \mathcal A_X$
which is linearly defined by $\emptyset\, \shuffle\, \omega:=\omega\, \shuffle\, \emptyset:=w$ and
\begin{equation}\label{eqn:shuffle product}
	a\omega\ \shuffle\ b\eta
	:=a(\omega\, \shuffle\, b\eta)+b(a\omega\, \shuffle\, \eta),
\end{equation}
for $a,b\in X_\Z$ and $\omega,\eta\in X_\Z^\bullet$.

\begin{defn}[{cf. \cite[Definition 1.4]{FK}}]\label{def:al,il,as,is}
A mould $M\in \ARI(\mathcal F)$ (resp. $\in\GARI(\mathcal F)$) is called {\it alternal} (resp. {\it symmetral}) if we have
\begin{align}\label{eqn:def of al,as}
	&M^{p+q}\bigl( (x_1,\ \dots,\ x_p)\shuffle(x_{p+1},\ \dots,\ x_{p+q}) \bigr) = 0 \\
	&\hspace{4.5cm}
	(\mbox{resp. } =M^{p}(x_1,\ \dots,\ x_p)M^{q}(x_{p+1},\ \dots,\ x_{p+q})) \nonumber
\end{align}
for $p,q\geq1$.
The $\Q$-linear space {$\ARI(\mathcal F)_\al$ (resp. $\GARI(\mathcal F)_\as$)} is defined to be the subset of moulds $M$ in $\ARI(\mathcal F)$ (resp. $\GARI(\mathcal F)$) which are alternal (resp. symmetral).
\end{defn}

To reformulate the notion of the alternality and the symmetrality,
we consider the following map.

\begin{defn}[{\cite[Definition 2.3]{Komi}}]\label{def:shmap}
The $\Q$-linear map
$\shmap: \mathcal{M}(\mathcal{F}) \rightarrow\mathcal{M}_2(\mathcal{F})$
is defined by
\begin{align*}
	\shmap(M)
	&:=\Bigl( M^{r+s}\bigl( (x_1,\ \dots,\ x_r)\shuffle(x_{r+1},\ \dots,\ x_{r+s}) \bigr) \Bigr)_{r,s\ge0}
\end{align*}
for $M\in\mathcal{M}(\mathcal{F})$.
\end{defn}

This map has the following algebraic property.

\begin{lem}[{\cite[Lemma 2.6]{Komi}, \cite[Lemma 5.1]{Sau}}]\label{lem:shmap is alg. hom.}
The map $\shmap$ is a $\Q$-algebra homomorphism.
\end{lem}

By using this map $\shmap$ and the tensor product (Definition \ref{def:tensor product}), we reformulate symmetral moulds and alternal moulds in terms of dimoulds as follows

\begin{prop}[{\cite[Proposition 2.4]{Komi}}]\label{prop:gp-like, Lie-like}
For a mould $M\in\mathcal{M}(\mathcal{F})$, the following equivalences hold:

{\rm (i).} $M\in\ARI(\mathcal F)_\al$
%$M\in\ARI(\mathcal F;\Gamma)$ is alternil
if and only if
$%\Longleftrightarrow
\shmap(M)=M\otimes \unitmould+\unitmould\otimes M$,

{\rm (ii).} $M\in\GARI(\mathcal F)_\as$
%\Longleftrightarrow
%$M\in\GARI(\mathcal F;\Gamma)$ is symmetril
if and only if
$\shmap(M)=M\otimes M$
and $M(\emptyset)=1$.
\end{prop}

%%%%%%%%%%%%%%%%%%%%%%%%%%%%
\subsection%{The maps $\sang$ and $\slang$}
{The singulator operator}
%In this subsection we
%take $\mathcal F$ to be family $\mathcal F_\Lau$ of Laurent series  and
We recall the definition of two maps $\sang$ and $\slang$,
which are required to compare Brown's and Ecalle's elements in next section.
In this subsection we take $\mathcal F$ to be the family $\mathcal F_\Lau$ of Laurent series
since denominators are required in their definitions.

\begin{defn}[{\cite[\S 5.4]{E-flex}}]
%We take $\mathcal F$ to be family $\mathcal F_\Lau$ of Laurent series.
The {\it singulator operator}
%\footnote{
%It seems that there is a sign error in the formula
%in  \cite[Page 78]{E-flex},
%it should be read as
%  \begin{align*}
%& 2(\sang.S)^{\bm{w}}=\sum_{\bm{a}w_{i}\bm{b}=\bm{w}}\mupaj^{\bm{a}}S^{w_{i}}\paj^{\bm{b}}
%    +\sum_{\bm{a}w_{i}\bm{b}=\bm{w}}\paj^{\bm{a}\rfloor}(\mathrm{neg}.S)^{\lceil w_{i}\rceil}\mupaj^{\lfloor\bm{b}}\\
% &   \quad  -\sum_{\bm{a}w_{i}\bm{b}w_{r}=\bm{w}}\paj^{\bm{a}\rfloor}(\mathrm{neg}.S)^{\lceil w_{i}\rceil}\mupaj^{\lfloor\bm{b}}P(|\bm{u}|)
%    +\sum_{w_{1}\bm{a}w_{i}\bm{b}=\bm{w}}\paj^{\bm{a}\rfloor}(\mathrm{neg}.S)^{\lceil w_{i}\rceil}\mupaj^{\lfloor\bm{b}}P(|\bm{u}|),
%  \end{align*}
%  which follows from \cite[(5.31)]{E-flex}.
%}
$$
\sang: {\mathcal M}(\mathcal F_\Lau)\to {\mathcal M}(\mathcal F_\Lau)
$$
is the map sending $M \in{\mathcal M}(\mathcal F_\Lau)$ to
%the mould given by
$$
\sang(M):=\frac{1}{2}\left( \id_{{\mathcal M}(\mathcal F_\Lau)} + \nega \circ \adari(\paj) \right) \left( \paj^{\times-1} \times M\times \paj \right).
$$
%for $M \in{\mathcal M}(\mathcal F_\Lau)$,
Here
$\paj\in\mathcal M(\mathcal F_\Lau)$ is defined by
$$
\paj(x_1,\dots,x_m)=\frac{1}{x_1(x_1+x_2)\cdots (x_1+\cdots+x_m)}.
$$
\end{defn}

\begin{eg}
For $A \in\ARI(\mathcal F_\Lau)$, we have
\begin{align*}
&(\sang(A))^0(\emptyset)=0, \\
&(\sang(A))^1(x_1)=\frac{1}{2}\left\{ A^1(x_1)+A^1(-x_1) \right\}, \\
&(\sang(A))^2(x_1,x_2)
=\frac{1}{2}\left\{ A^2(x_1,x_2) +A^2(-x_1,-x_2) \right\} \\
&\qquad +\frac{1}{2x_2}\left\{A^1(x_1) -A^1(-x_1)\right\} -\frac{1}{2x_1}\left\{ A^1(x_2) -A^1(-x_2) \right\} \\
&\ \qquad +\frac{x_1}{2x_2(x_1+x_2)}A^1(-x_1) -\frac{x_2}{2x_1(x_1+x_2)}A^1(-x_2)  + \frac{x_2-x_1}{2x_1x_2}A^1(-x_1-x_2).
\end{align*}
\end{eg}

\begin{rem}\label{rem: proof for conj}
It seems that there is a sign error in the formula in  \cite[Page 78]{E-flex},
it should be read as
%  In \cite[Page 78]{E-flex}, it is written that
%\begin{align*}
%2(\sang.S)^{\bm{w}} & =+\sum_{\bm{a}w_{i}\bm{b}=\bm{w}}\mupaj^{\bm{a}}S^{w_{i}}\paj^{\bm{b}}\\
% & \quad+\sum_{\bm{a}w_{i}\bm{b}=\bm{w}}\paj^{\bm{a}\rfloor}(\mathrm{neg}.S)^{\lceil w_{i}\rceil}\mupaj^{\lfloor\bm{b}}\\
% & \quad+\sum_{\bm{a}w_{i}\bm{b}w_{r}=\bm{w}}\paj^{\bm{a}\rfloor}(\mathrm{neg}.S)^{\lceil w_{i}\rceil}\mupaj^{\lfloor\bm{b}}P(|\bm{u}|)\\
% & \quad-\sum_{w_{1}\bm{a}w_{i}\bm{b}=\bm{w}}\paj^{\bm{a}\rfloor}(\mathrm{neg}.S)^{\lceil w_{i}\rceil}\mupaj^{\lfloor\bm{b}}P(|\bm{u}|).
%\end{align*}
%  It seems this equation is wrong, and the correct formula is
  \begin{align*}
  2(\sang.S)^{\bm{w}} & =+\sum_{\bm{a}w_{i}\bm{b}=\bm{w}}\mupaj^{\bm{a}}S^{w_{i}}\paj^{\bm{b}}\\
   & \quad+\sum_{\bm{a}w_{i}\bm{b}=\bm{w}}\paj^{\bm{a}\rfloor}(\mathrm{neg}.S)^{\lceil w_{i}\rceil}\mupaj^{\lfloor\bm{b}}\\
   & \quad  -\sum_{\bm{a}w_{i}\bm{b}w_{r}=\bm{w}}\paj^{\bm{a}\rfloor}(\mathrm{neg}.S)^{\lceil w_{i}\rceil}\mupaj^{\lfloor\bm{b}}P(|\bm{u}|)\\
   & \quad +\sum_{w_{1}\bm{a}w_{i}\bm{b}=\bm{w}}\paj^{\bm{a}\rfloor}(\mathrm{neg}.S)^{\lceil w_{i}\rceil}\mupaj^{\lfloor\bm{b}}P(|\bm{u}|),
  \end{align*}
  which follows from \cite[(5.31)]{E-flex}.
\end{rem}

%%%%%%%%%%%%%%%%%%%%%%%%%%%%
%\subsection{The mould $\pal$}
%\Add{Insert definitions of $dupal$ and $pal$.}
%In this subsection we recall the definition of moulds $\pal\in\GARI(\mathcal F_\Lau)$ in \cite{S-ARIGARI}.
%We start with the definition of moulds $\dupal\in\ARI(\mathcal F_\Lau)$.

\begin{defn}[{\cite[Definition 4.2.2]{S-ARIGARI}}]
The mould $\dupal\in\ARI(\mathcal F_\Lau)$ is defined by
\begin{align*}
\dupal^m(\vecx_m)
:=\frac{B_m}{m!}\frac{1}{x_1\cdots x_m}\sum_{k=0}^{m-1}(-1)^k\binom{m-1}{k}x_{k+1},
\end{align*}
for $m\ge1$.
Here, symbols $B_m$ are the Bernoulli numbers which are defined by $\frac{x}{e^x-1}=\sum_{m\ge0}\frac{B_m}{m!}x^m$.
\end{defn}

\begin{eg}
Note that we have $\dupal^{2n+1}(\vecx_{2n+1})=0$ for $n\ge1$.
We have
\begin{align*}
&\dupal^1(\vecx_1)=-\frac{1}{2},
\quad \dupal^2(\vecx_2)=\frac{x_1-x_2}{12x_1x_2},
\quad \dupal^4(\vecx_4)=-\frac{x_1-3x_2+3x_3-x_4}{720x_1x_2x_3x_4}.
\end{align*}
\end{eg}

To show that the mould is alternal, we prepare the following.

\begin{lem}\label{lem:alternal mould by inductive definition}
Let $A\in\ARI(\mathcal F_\Lau)$ with the condition
\begin{equation}\label{eqn:alternal mould by inductive definition}
A^m(x_1,\dots,x_m)=A^{m-1}(x_1,\dots,x_{m-1}) - A^{m-1}(x_2,\dots,x_m)
\end{equation}
for $m\ge2$.
Then the mould $A$ is alternal.
\end{lem}
\begin{proof}
It is sufficient to prove
\begin{equation}\label{eqn:A is alternal}
A^{m+n}\bigl( (x_1,\dots,x_m)\shuffle (x_{m+1},\dots,x_{m+n}) \bigr) = 0
\end{equation}
for $m,n\ge1$.
We prove this by induction on $m+n$.
When $m=n=1$, by \eqref{eqn:alternal mould by inductive definition}, we have
\begin{align*}
A^2\bigl( (x_1)\shuffle (x_2) \bigr)
=A^2(x_1,x_2) + A^2(x_2,x_1)
=\{A^1(x_1)-A^1(x_2)\} + \{A^1(x_2)-A^1(x_1)\}
=0.
\end{align*}
When $m=1$ and $n=2$, by \eqref{eqn:alternal mould by inductive definition}, we calculate
\begin{align*}
&A^3\bigl( (x_1)\shuffle (x_2,x_3) \bigr) \\
&=A^3(x_1,x_2,x_3) + A^3(x_2,x_1,x_3) + A^3(x_2,x_3,x_1) \\
&=\{A^2(x_1,x_2)-A^2(x_2,x_3)\} + \{A^2(x_2,x_1)-A^2(x_1,x_3)\} + \{A^2(x_2,x_3)-A^2(x_3,x_1)\} \\
&=0.
\end{align*}
The same proof holds for $m=2$ and $n=1$. % as for $m=1$ and $n=2$.
We consider the case for $m,n\ge2$.
For our simplicity, we put
$$
x_{i,j}:=\left\{\begin{array}{ll}
(x_i,\dots,x_j) & (i\le j), \\
\emptyset & (i>j).
\end{array}\right.
$$
Then we have
\begin{align*}
&A^{m+n}\bigl( x_{1,m}\shuffle\, x_{m+1,m+n} \bigr) \\
&=A^{m+n}\bigl( x_{1,m-1}\shuffle\, x_{m+1,m+n}, x_m \bigr)
+ A^{m+n}\bigl( x_{1,m}\shuffle\, x_{m+1,m+n-1}, x_{m+n} \bigr) \\
%2-equality
&=A^{m+n}\bigl( x_1, x_{2,m-1}\shuffle\, x_{m+1,m+n}, x_m \bigr)
+ A^{m+n}\bigl( x_{m+1}, x_{1,m-1}\shuffle\, x_{m+2,m+n}, x_m \bigr) \\
&\quad+ A^{m+n}\bigl( x_1, x_{2,m}\shuffle\, x_{m+1,m+n-1}, x_{m+n} \bigr)
+ A^{m+n}\bigl( x_{m+1}, x_{1,m}\shuffle\, x_{m+2,m+n-1}, x_{m+n} \bigr).
\end{align*}
By using \eqref{eqn:alternal mould by inductive definition}, we calculate
\begin{align*}
&A^{m+n}\bigl( x_{1,m}\shuffle\, x_{m+1,m+n} \bigr) \\
&=A^{m+n-1}\bigl( x_1, x_{2,m-1}\shuffle\, x_{m+1,m+n} \bigr)
-A^{m+n-1}\bigl( x_{2,m-1}\shuffle\, x_{m+1,m+n}, x_m \bigr) \\
&\quad+ A^{m+n-1}\bigl( x_{m+1}, x_{1,m-1}\shuffle\, x_{m+2,m+n} \bigr)
-A^{m+n-1}\bigl( x_{1,m-1}\shuffle\, x_{m+2,m+n}, x_m \bigr) \\
&\qquad+ A^{m+n-1}\bigl( x_1, x_{2,m}\shuffle\, x_{m+1,m+n-1} \bigr)
- A^{m+n-1}\bigl( x_{2,m}\shuffle\, x_{m+1,m+n-1}, x_{m+n} \bigr) \\
&\qquad\quad+ A^{m+n-1}\bigl( x_{m+1}, x_{1,m}\shuffle\, x_{m+2,m+n-1} \bigr)
-A^{m+n-1}\bigl( x_{1,m}\shuffle\, x_{m+2,m+n-1}, x_{m+n} \bigr) \\
%2-equality
&=A^{m+n-1}\bigl( x_{1,m-1}\shuffle\, x_{m+1,m+n} \bigr)
-A^{m+n-1}\bigl( x_{2,m}\shuffle\, x_{m+1,m+n} \bigr) \\
&\quad+ A^{m+n-1}\bigl( x_{1,m}\shuffle\, x_{m+1,m+n-1} \bigr)
-A^{m+n-1}\bigl( x_{1,m}\shuffle\, x_{m+2,m+n} \bigr).
\end{align*}
So, by induction hypothesis, we get $A^{m+n}\bigl( x_{1,m}\shuffle\, x_{m+1,m+n} \bigr)=0$ for $m,n\ge2$.
Hence, we obtain \eqref{eqn:A is alternal}, that is, the mould $A$ is alternal.
\end{proof}

\begin{prop}\label{prop:dupal is alternal mould}
The mould $\dupal$ is in $\ARI(\mathcal F_\Lau)_\al$.
\end{prop}
\begin{proof}
Let $A\in\ARI(\mathcal F_\Lau)$ be the mould defined by $A^m(\vecx_m):=\sum_{k=0}^{m-1}(-1)^k\binom{m-1}{k}x_{k+1}$ for $m\ge1$.
Then the mould $A$ satisfies \eqref{eqn:alternal mould by inductive definition},
so by Lemma \ref{lem:alternal mould by inductive definition},  $A$ is an alternal mould.
%Note that, for $r\geq1$, moulds $A_1,\dots,A_r\in\ARI(\mathcal F_\Lau)_\al$ and $a_1,\dots,a_r\in\Q$, linear sum $B=a_1A_1+\cdots+a_rA_r$ is alternal mould and the mould $\left( \frac{1}{x_1\dots x_m} B^m(\vecx_m)\right)_{m\ge0}$ is also alternal mould.
Because $\dupal^m(\vecx_m)=\frac{B_m}{m!}\frac{1}{x_1\cdots x_m}A^m(\vecx_m)$ for $m\ge1$, we have
\begin{align*}
&\dupal^{m+n}((x_1,\dots,x_m)\shuffle(x_{m+1},\dots,x_{m+n})) \\
&=\frac{B_{m+n}}{(m+n)!}\frac{1}{x_1\cdots x_{m+n}}A^{m+n}((x_1,\dots,x_m)\shuffle(x_{m+1},\dots,x_{m+n})) \\
&=0.
\end{align*}
Hence, the mould $\dupal$ is alternal.
\end{proof}

\begin{defn}
For $U\in\mathcal{M}(\mathcal F)$, $M\in\mathcal{M}(\mathcal F)$ and $N\in\mathcal{M}_2(\mathcal F)$,
we define
$U\cdot M\in\mathcal{M}(\mathcal F)$ and $U\cdot N\in\mathcal{M}_2(\mathcal F)$ by
%its actions on $\mathcal{M}(\mathcal F)$ and  $\mathcal{M}_2(\mathcal F)$
% $U\cdot*:\mathcal{M}(\mathcal F)\rightarrow\mathcal{M}(\mathcal F)$ and $U\cdot*:\mathcal{M}_2(\mathcal F)\rightarrow\mathcal{M}_2(\mathcal F)$ by
\begin{align*}
(U\cdot M)^{l(\omega)}(\omega)&:=U^{l(\omega)}(\omega)M^{l(\omega)}(\omega), \\
(U\cdot N)^{l(\omega), l(\eta)}(\omega;\eta)&:=U^{l(\omega), l(\eta)}(\omega,\eta)N^{l(\omega), l(\eta)}(\omega;\eta),
\end{align*}
for $\omega,\eta\in X^\bullet$.
%, $M\in\mathcal{M}(\mathcal F)$ and $N\in\mathcal{M}_2(\mathcal F)$.
\end{defn}

%For $U\in\ARI(\mathcal F)$, we consider the following condition:
%\begin{equation}\label{cond:A.1}
%	\left\{\begin{array}{l}
%		\mbox{$U^{l(\omega_1) + l(\omega_2)}(\omega_1,\omega_2)=U^{l(\omega_1)}(\omega_1)+U^{l(\omega_2)}(\omega_2)$\qquad ($\omega_1,\omega_2\in X^\bullet$),} \\
%		\mbox{$U^{l(\omega)}(\omega)\neq0$ \qquad ($\omega\in X^\bullet\setminus\{\emptyset\}$).}
%	\end{array}\right.
%\end{equation}

\begin{lem}
Suppose that $U\in\ARI({\mathcal F})$ satisfies %the following condition
\begin{equation}\label{cond:A.1}
	\left\{\begin{array}{l}
		\mbox{$U^{l(\omega_1) + l(\omega_2)}(\omega_1,\omega_2)=U^{l(\omega_1)}(\omega_1)+U^{l(\omega_2)}(\omega_2)$\qquad ($\omega_1,\omega_2\in X^\bullet$),} \\
		\mbox{$U^{l(\omega)}(\omega)\neq0$ \qquad ($\omega\in X^\bullet\setminus\{\emptyset\}$).}
	\end{array}\right.
\end{equation}
Then the following assertions hold:
%\begin{description}
\item[\rm (i)]
$U\cdot\mathpzc{Sh}(M)=\mathpzc{Sh}(U\cdot M)$\qquad $(M\in\mathcal M(\mathcal F))$,
\item[\rm (ii)]
$U\cdot(M\otimes N) = (U\cdot M)\otimes N + M\otimes(U\cdot N)$\qquad $(M,N\in\mathcal M(\mathcal F))$,
\item[\rm (iii)]
For $M,N\in\mathcal M(\mathcal F)$ (or for $M,N\in\mathcal M_2(\mathcal F)$), we have
$$
U \cdot(M \times N) = (U\cdot M)\times N + M\times(U\cdot M).
$$
%\end{description}
\end{lem}

\begin{proof}
It can be proved by a straightforward calculation.
\end{proof}

\begin{prop}[{cf. \cite[Th\'{e}or\`{e}me IV.2]{Cre}}]\label{prop:Cresson's proposition}
Let $U\in\ARI(\mathcal F)$ be a mould satisfying the conditions (\ref{cond:A.1}).
Suppose $A\in\ARI(\mathcal F)$ and $S\in\GARI(\mathcal F)$ satisfy $U\cdot S=A\times S$ (or $=S\times A$).
Then the following assertions are equivalent:
%\begin{description}
	\item[\rm (i)] $A$ is alternal.
	\item[\rm (ii)] $S$ is symmetral.
%\end{description}
\end{prop}

\begin{proof}
\underline{(i) $\Rightarrow$ (ii)}:
For $\omega,\eta\in X^{\bullet}\setminus\{\emptyset\}$, we prove $\mathpzc{sh}(S)^{l(\omega), l(\eta)}(\omega;\eta)=(S\otimes S)^{l(\omega), l(\eta)}(\omega;\eta)$ by induction on $l(\omega)+l(\eta)\geq2$.
It is easy to show the case of $l(\omega)=l(\eta)=1$.
For $l(\omega)+l(\eta)\geq3$, we have
\begin{align}\label{eqn:calculation of Ush(S)}
&(U\cdot\mathpzc{sh}(S))^{l(\omega), l(\eta)}(\omega;\eta) \\
&=\mathpzc{sh}(U\cdot S)^{l(\omega), l(\eta)}(\omega;\eta) \nonumber \\
&=\mathpzc{sh}(A\times S)^{l(\omega), l(\eta)}(\omega;\eta) \nonumber \\
&=(\mathpzc{sh}(A)\times \mathpzc{sh}(S))^{l(\omega), l(\eta)}(\omega;\eta) \nonumber \\
&=((A\otimes 1_{\mathcal{M}(\mathcal F)} + 1_{\mathcal{M}(\mathcal F)}\otimes A)\times \mathpzc{sh}(S))^{l(\omega), l(\eta)}(\omega;\eta) \nonumber \\
&=\sum_{\substack{\omega=\omega'\omega'' \\ \eta=\eta'\eta''}}
(A\otimes \unitmould + \unitmould\otimes A)^{l(\omega'), l(\eta')}(\omega';\eta') \mathpzc{sh}(S)^{l(\omega''), l(\eta'')}(\omega'';\eta''). \nonumber
\end{align}
Because we have $(A\otimes \unitmould + \unitmould\otimes A)^{0,0}(\emptyset;\emptyset)=0$, in the sum of the last member of \eqref{eqn:calculation of Ush(S)}, two parameters $\omega',\eta'$ run over $(\omega',\eta') \neq (\emptyset,\emptyset)$, that is, we obtain
$$
l(\omega'')+l(\eta'') = \{ l(\omega) - l(\omega') \} + \{ l(\eta) - l(\eta') \} < l(\omega)+l(\eta).
$$
So by induction hypothesis, we have
\begin{align*}
(U\cdot\mathpzc{sh}(S))^{l(\omega), l(\eta)}(\omega;\eta)
&=((A\otimes \unitmould + \unitmould\otimes A)\times (S\otimes S))^{l(\omega), l(\eta)}(\omega;\eta) \\
&=((A\times S)\otimes S + S\otimes (A\times S))^{l(\omega), l(\eta)}(\omega;\eta) \\
&=((U\cdot S)\otimes S + S\otimes (U\cdot S))^{l(\omega), l(\eta)}(\omega;\eta) \\
&=(U\cdot (S\otimes S))^{l(\omega), l(\eta)}(\omega;\eta).
\end{align*}
Because $\omega,\eta\neq\emptyset$, by the condition (\ref{cond:A.1}), we obtain $\mathpzc{sh}(S)^{l(\omega), l(\eta)}(\omega;\eta)=(S\otimes S)^{l(\omega), l(\eta)}(\omega;\eta)$, that is, $S$ is symmetral.

\bigskip
\noindent
\underline{(ii) $\Rightarrow$ (i)}:
By $S\in\GARI(\mathcal F)$, we get $A=S^{\times-1}\times (U\cdot S)$, so we have
\begin{align*}
\mathpzc{Sh}(A)
&=\mathpzc{Sh}\left( S^{\times-1}\times (U\cdot S) \right) \\
&=\mathpzc{Sh}\left( S^{\times-1} \right)\times \mathpzc{Sh}\left( U\cdot S \right) \\
&=\left( S^{\times-1}\otimes S^{\times-1} \right)\times \left\{U \cdot\mathpzc{Sh}(S)\right\} \\
&=\left( S^{\times-1}\otimes S^{\times-1} \right)\times \left\{U \cdot(S \otimes S)\right\} \\
&=\left( S^{\times-1}\otimes S^{\times-1} \right)\times \left\{(U \cdot S)\otimes S + S\otimes (U\cdot S)\right\} \\
&=\left( S^{\times-1}\times (U\cdot S)\right)\otimes  \left( S^{\times-1} \times S\right)
+\left( S^{\times-1} \times S \right)\otimes \left( S^{\times-1}\times (U\cdot S)\right) \\
&=A\otimes \unitmould + \unitmould\otimes A.
\end{align*}
Hence $A$ is alternal.
\end{proof}

\begin{defn}[{\cite[(4.2.5)]{S-ARIGARI}}]
We define the mould $\pal\in\GARI(\mathcal F_\Lau)$ by
$$
\dur\cdot \pal=\pal \times \dupal
$$
where the mould $\dur\in\ARI(\mathcal F_\Lau)$ is defined by
$$\dur^m(\vecx_m)=x_1+\cdots+x_m \text{ for }m\ge1.$$
\end{defn}

\begin{eg}
We have
\begin{align*}
\pal^1(\vecx_1)=-\frac{1}{2x_1},
\qquad \pal^2(\vecx_2)=\frac{x_1+2x_2}{12x_1x_2(x_1+x_2)},
\qquad \pal^3(\vecx_3)=-\frac{1}{24x_1x_3(x_1+x_2)},
\end{align*}
\end{eg}

\begin{prop}[{\cite[Theorem 4.3.4]{S-ARIGARI}}]
The mould $\pal$ is in $\GARI(\mathcal F_\Lau)_\as$.
\end{prop}
\begin{proof}
By Proposition \ref{prop:dupal is alternal mould} and Proposition \ref{prop:Cresson's proposition}, we obtain the claim.
\end{proof}

\begin{defn}[{\cite[\S 5.5]{E-flex}}]
For $r\in\N$,
the map %{\it singulator operator}
$$
\slang_r: \ARI(\mathcal F_\Lau)\to \ARI(\mathcal F_\Lau)
$$
is defined by
$$
\slang_r(A):=\adari(\pal) \circ \leng_r \circ \adari(\pal)^{-1 }\circ \sang(A),
$$
for $A \in\ARI(\mathcal F_\Lau)$,
where the $\Q$-linear map $\leng_r$ on $\ARI(\mathcal F)$ is defined by
$$
(\leng_r(M))^m(\vecx_m):=\delta_{m,r}M^m(\vecx_m)
$$
for $m\ge0$.
Here, the symbol $\delta_{m,r}$ is the Kronecker delta.
\end{defn}

By definition, we have
$\sang=\sum_{r\in\N} \slang_r$.

\begin{eg}
For $A \in\ARI(\mathcal F_\Lau)$, we have
\begin{align*}
&(\slang_r(A))^0(\emptyset)=0, \\
&(\slang_r(A))^1(x_1)=\frac{\delta_{r,1}}{2}\left\{ A^1(x_1)+A^1(-x_1) \right\}, \\
&(\slang_r(A))^2(x_1,x_2) \\
&=\frac{\delta_{r,2}}{2}\left[ A^2(x_1,x_2)+A^2(-x_1,-x_2)+\frac{x_1-x_2}{2x_1x_2}\{A^1(x_1+x_2)-A^1(-x_1-x_2)\} \right. \\
&\hspace{1.1cm} \left. +\frac{x_1+2x_2}{2x_2(x_1+x_2)}\{A^1(x_1)-A^1(-x_1)\} -\frac{2x_1+x_2}{2x_1(x_1+x_2)}\{A^1(x_2)-A^1(-x_2)\} \right] \\
%4-line
&\quad+\frac{\delta_{r,1}}{2}\left[ \frac{x_1}{x_2(x_1+x_2)}\{A^1(x_1)+A^1(-x_1)\} -\frac{x_2}{x_1(x_1+x_2)}\{A^1(x_2)+A^1(-x_2)\} \right. \\
&\quad\hspace{6.1cm} \left. -\frac{x_1-x_2}{x_1x_2}\{A^1(x_1+x_2)+A^1(-x_1-x_2)\} \right].
\end{align*}
\end{eg}

%%%%%%%%%%%%%%%%%%%%%%%%%%%%%%%%%%%%%%%%%%%%%%%%%%%%%%%%%%%%%%%%%%%%%%
\section{On polar solutions}\label{sec: On polar solutions}
%\section{Interpretations of Brown's polar solutions in mould Theory}\label{Interpretations of Brown's polar solutions in mould Theory}

In this section, we reinterpret Brown's polar solutions,
the elements $\psi_{2n+1}$ ($n\geq 1$) and $\psi_{-1}$
%introduced by Brown \cite{B-anatomy},
in terms of Ecalle's mould theory (cf. \S \ref{Preparation}).
In \S \ref{sec: Brown's polar solutions}, we recall Brown's construction of the polar solutions to the double shuffle relations modulo products presented in \cite{B-anatomy}.
In \S \ref{sec: Interpretations in mould theory},
we provide mould-theoretic interpretations of these polar solutions (Theorems \ref{thm 2n+1} and \ref{thm -1}).

%%%%%%%%%%%%%%%%%%%%%%%%%%%%%%%%%%%%%%%%%%%%%%%%%%%%%%%%%%%%%%%%%%%%%%
\subsection{Brown's polar solutions}\label{sec: Brown's polar solutions}
We start with the notation in \cite[(10.1)]{B-anatomy}.
For any sets of indices $A,B\subset \{0,1,\dots,d\}$, we use the notation\footnote{This notation is introduced in \cite[(10.1)]{B-anatomy}.}
$$
x_{A,B}=\prod_{a\in A, b\in B}(x_a-x_b).
$$
If $A$ or $B$ is the empty set, then we put $x_{A,B}:=1$.
%We sometimes use $x_0:=0$.

We recall the definitions of Brown's polar solutions.
\begin{defn}[{\cite[Definition 10.1]{B-anatomy}}]
For $n\geq1$, the element $\psi_{2n+1}=(\psi^{(d)}_{2n+1})_d\in \swap(\ARI(\mathcal F_\Lau))$ is defined by\footnote{It seems the signature in \cite{B-anatomy} is incorrect.}
\begin{eqnarray}
&  \psi^{(d)}_{2n+1} (x_1,\dots,x_d) =  {1 \over 2} \sum_{i=1}^d \Big({ (x_i - x_{i-1})^{2n} \over x_{\{0,\ldots, i-2\},\{i-1\}}\, x_{\{i+1,\ldots,d\},\{i\}}}  +  { x_d^{2n} \over x_{\{1,\ldots, i-1\},\{0\}}\, x_{\{i,\ldots,d-1\},\{d\}}}\Big)   \nonumber \\
  &  - {1\over 2 }    \sum_{i=1}^{d-1}\Big( { (x_1-x_d)^{2n} \over x_{\{2,\ldots, i\},\{ 1\} } \, x_{\{i+1,\ldots, d-1,0\}, \{d\} }}   + { x_{d-1}^{2n} \over x_{\{d, 1,\ldots, i-1\}, \{0\}} \, x_{\{i,\ldots, d-2\}, \{d-1\}}} \Big)
  \nonumber
\end{eqnarray}
where $x_0=0$.
%, and let $\psi^{(0)}_{2n+1}=0$.
\end{defn}
The following theorem is given by Brown.
\begin{thm}[{\cite[Theorem 10.2]{B-anatomy}}]
For $n\geq1$, we have $\psi_{2n+1}\in \swap (\ARI_{\underline{\al\ast\il}}(\mathcal F_\Lau))$.
\end{thm}

To explain another polar solution introduced by Brown, we review the following definition.
\begin{defn}[{\cite[Definition 10.4]{B-anatomy}}]
\begin{enumerate}
\item We call the following graph $g_n$ with vertices labelled from the set $\{i,i+1,\dots,i+n\}$ a {\it bunch of $n$ grapes}:
\begin{center}
	\begin{tikzpicture}
	\coordinate [label=right:\quad$i$] (i) at (0,0);
	\coordinate [label=below:$i+1$] (i+1) at (-2,-1);
	\coordinate [label=below:$i+2$] (i+2) at (-1,-1);
	\coordinate  (i+n-1) at (1,-1);
	\coordinate [label=below:$i+n$] (i+n) at (2,-1);
	\foreach \P  in {i,i+1,i+2,i+n-1,i+n} \fill[black] (\P) circle (0.1);
	\draw [ultra thick] (i+1)--(i)--(i+n-1);
	\draw [ultra thick] (i+2)--(i)--(i+n);
	\draw (0,-1)node[]{$\cdots$};
	\draw (-3,-0.5)node[]{$g_n=$};
	\end{tikzpicture}
\end{center}
We call the vertex $i$ the {\it stalk}, and the vertices $i+1,\dots,i+n$ the {\it grapes}.%Brownの論文では$i,\dots,i+n$をgrapeと呼ぶと書いている
\footnote{
In \cite{B17b}, $i, \dots, i+n$ are referred to as the vertices, but this may be a typo.
}
\item A {\it vine} is a rooted tree whose vertices have distinct labels $\{0,1,\dots,n\}$, where $0$ denotes the root vertex, obtained by grafting bunches of grapes.
Any vine $v$ is uniquely represented by a sequence $v=g_{n_1}\cdots g_{n_k}$, where the stalk of each bunch of grapes $g_{n_i}$ is grafted to the grape with the highest label of the vine $g_{n_1}\cdots g_{n_{i-1}}$.
For any vine $v=g_{n_1}\cdots g_{n_k}$, we define the {\it height} of $v$ by $h(v):=k$.
\item Denote $\mathcal V_d$ the set of all vines with $d$ grapes.\footnote{Note that the height $h(v)$ of $v\in\mathcal V_d$ satisfies $1\leq h(v)\leq d$.}
\end{enumerate}
\end{defn}

For any vine $v$, we use the notation\footnote{This notation is introduced in \cite[(10.3)]{B-anatomy}.}
$$
x_v=\prod_{(i,j)\in E(v)}(x_j-x_i)
$$
where $x_0=0$ and the product is over edges $(i,j)$ with $i<j$.

\begin{defn}[{\cite[Definition 10.7]{B-anatomy}}]
We define the element $\psi_{-1}=(\psi^{(d)}_{-1})_d\in \swap(\ARI(\mathcal F_\Lau))$ by
\begin{equation*}
\psi^{(d)}_{-1}(x_1,\dots,x_d)
=\sum_{v\in\mathcal V_d}\frac{(-1)^{h(v)+1}}{h(v)}\frac{1}{x_vx_d},
\end{equation*}
where the sum is over vines with $d$ grapes.
\end{defn}

\begin{lem}
For $d\geq1$, we have
\[
\psi_{-1}^{(d)}(x_{1},\dots,x_{d})=\frac{1}{x_{d}}\sum_{h=1}^d\frac{(-1)^{h+1}}{h}\sum_{0=i_{0}<i_{1}<\cdots<i_{h}=d}\prod_{s=0}^{h-1}\frac{1}{(x_{i_{s}+1}-x_{i_{s}})\cdots(x_{i_{s+1}}-x_{i_{s}})}
\]
where $x_0=0$.
\end{lem}

\begin{proof}
%For $v\in \mathcal V_d$, the height $h=h(v)$ of $v$ satisfies $1\le h\le d$, and there are exist $n_1,\dots,n_h\in\N$ such that $v=g_{n_1}\cdots g_{n_h}$.
Note that $v$ is in $\mathcal V_d$ if and only if there exist $h\in\N$ with $1\le h\le d$ and $n_1,\dots,n_h\in\N$ such that $v=g_{n_1}\cdots g_{n_h}$.
On these conditions, we put $i_0=0$, $i_s=n_1+\cdots+n_s$ ($s=1,\dots,h$), then we have
$$
x_v=\prod_{s=0}^{h-1} \left\{ (x_{i_{s}+1}-x_{i_{s}})\cdots(x_{i_{s+1}}-x_{i_{s}}) \right\}.
$$
The above $h$ means the height $h(v)$ of $v$, so we obtain the claim.
\end{proof}

The following theorem is given by Brown.
\begin{thm}[{\cite[Theorem 10.8]{B-anatomy}}]
We have $\psi_{-1}\in \swap (\ARI_{\underline{\al\ast\il}}(\mathcal F_\Lau))$.
\end{thm}

%%%%%%%%%%%%%%%%%%%%%%%%%%%%%%%%%%%%%%%%%%%%%%%%%%%%%%%%%%%%%%%%%%%%%%
\subsection{Interpretations in mould theory}\label{sec: Interpretations in mould theory}
In this subsection, we give interpretations of Brown's polar solutions
%(see \S \ref{sec: Brown's polar solutions} for definitions)
$\psi_{2n+1}$ and $\psi_{-1}$
in mould theory.

\begin{thm}\label{thm 2n+1}
For $n\geq1$, we have
\[
	\psi_{2n+1}^\sharp = \sang(\sa_{2n+1}) \qquad \left(= \sum_{r=1}^{\infty} \slang_r(\sa_{2n+1}) = \sum_{r=1}^{\infty}\sa^\bullet_{\binom{2n+1}{r}} \right)
\]
with %$\sa_{2n+1}^{(r)}=0$ for $r\neq 1$ and
\begin{equation}\label{eq:sa}
\sa_{s}(u)=u^{s-1}
\end{equation}
(consult \cite[(9.12)]{E-flex} for $\sa^\bullet_s\in\ARI_{\al/\al}(\mathcal F_\ser)\cap
\BIMU_1$).
Here $\sharp$ is the map defined by
$$
f(x_1,\dots, x_r)\overset{\sharp}{\mapsto}f(x_1,x_1+x_2,\dots, x_1+\cdots+x_r).
$$
in \cite[(4.16)]{B-anatomy}.
\end{thm}

\begin{proof}
By Remark \ref{rem: proof for conj}, we have
\begin{align*}
&2\sang(S)(u_{1},\dots,u_{d}) \\
& =+\sum_{i=1}^{d}\mupaj(u_{1},\dots,u_{i-1})S(u_{i})\paj(u_{i+1},\dots,u_{d})\\
& \quad+\sum_{i=1}^{d}\paj(u_{1},\dots,u_{i-1})S(-(u_{1}+\cdots+u_{d}))\mupaj(u_{i+1},\dots,u_{d})\\
& \quad-\sum_{i=1}^{d-1}\paj(u_{1},\dots,u_{i-1})S(-(u_{1}+\cdots+u_{d-1}))\mupaj(u_{i+1},\dots,u_{d-1})\frac{1}{u_{1}+\cdots+u_{d}}\\
 & \quad+\sum_{i=2}^{d}\paj(u_{2},\dots,u_{i-1})S(-(u_{2}+\cdots+u_{d}))\mupaj(u_{i+1},\dots,u_{d})\frac{1}{u_{1}+\cdots+u_{d}}.
\end{align*}
Note that $\paj$ and $\mupaj$ is given by
\[
\paj(u_{1},\dots,u_{k})=\frac{1}{u_{1}(u_{1}+u_{2})\cdots(u_{1}+u_{2}+\cdots+u_{k})}
\]
and
\[
\mupaj(u_{1},\dots,u_{k})=\frac{(-1)^{k}}{u_{k}(u_{k}+u_{k-1})\cdots(u_{k}+u_{k-1}+\cdots+u_{1})}.
\]
Put $x_{j}=u_{1}+\cdots+u_{j}$ for $j=0,\dots,d$.
Then,
\begin{align*}
& \sum_{i=1}^{d}\mupaj(u_{1},\dots,u_{i-1})S(u_{i})\paj(u_{i+1},\dots,u_{d})\\
 & =\sum_{i=1}^{d}\frac{S(x_{i}-x_{i-1})}{\left(\prod_{j=0}^{i-2}(x_{j}-x_{i-1})\right)\left(\prod_{j=i+1}^{d}(x_{j}-x_{i})\right)}\\
 & =\sum_{i=1}^{d}\frac{S(x_{i}-x_{i-1})}{x_{\{0,\dots,i-2\},\{i-1\}}x_{\{i+1,\dots,d\},\{i\}}},
\end{align*}
\begin{align*}
& \sum_{i=1}^{d}\paj(u_{1},\dots,u_{i-1})S(-(u_{1}+\cdots+u_{d}))\mupaj(u_{i+1},\dots,u_{d}) \\
& =\sum_{i=1}^{d}\frac{S(-x_{d})}{\left(\prod_{j=1}^{i-1}x_{j}\right)\left(\prod_{j=i}^{d-1}(x_{j}-x_{d})\right)}\\
 & =\sum_{i=1}^{d}\frac{S(-x_{d})}{x_{\{1,\dots,i-1\},\{0\}}x_{\{i,\dots,d-1\},\{d\}}},
\end{align*}
\begin{align*}
& -\sum_{i=1}^{d-1}\paj(u_{1},\dots,u_{i-1})S(-(u_{1}+\cdots+u_{d-1}))\mupaj(u_{i+1},\dots,u_{d-1})\frac{1}{u_{1}+\cdots+u_{d}}\\
& =-\sum_{i=1}^{d-1}\frac{S(-x_{d-1})}{\left(\prod_{j=1}^{i-1}x_{j}\right)\left(\prod_{j=i}^{d-2}(x_{j}-x_{d-1})\right)x_{d}}\\
 & =-\sum_{i=1}^{d-1}\frac{S(-x_{d-1})}{x_{\{d,1,\dots,i-1\},\{0\}}x_{\{i,\dots,d-2\},\{d-1\}}},
\end{align*}
\begin{align*}
& \sum_{i=2}^{d}\paj(u_{2},\dots,u_{i-1})S(-(u_{2}+\cdots+u_{d}))\mupaj(u_{i+1},\dots,u_{d})\frac{1}{u_{1}+\cdots+u_{d}} \\
& =\sum_{i=2}^{d}\frac{S(x_{1}-x_{d})}{\left(\prod_{j=2}^{i-1}(x_{j}-x_{1})\right)\left(\prod_{j=i}^{d-1}(x_{j}-x_{d})\right)x_{d}}\\
 & =-\sum_{i=2}^{d}\frac{S(x_{1}-x_{d})}{x_{\{2,\dots,i-1\},\{1\}}x_{\{i,\dots,d-1,0\},\{d\}}}\\
 & =-\sum_{i=1}^{d-1}\frac{S(x_{1}-x_{d})}{x_{\{2,\dots,i\},\{1\}}x_{\{i+1,\dots,d-1,0\},\{d\}}}.
\end{align*}
Thus,
\begin{align*}
2\sang(\sa_{2n+1})(u_{1},\dots,u_{d}) & =2\psi_{2n+1}(x_{1},\dots,x_{d})\\
 & =2\psi_{2n+1}^{\#}(u_{1},\dots,u_{d}),
\end{align*}
which implies $\sang(\sa_{2n+1}) = \psi_{2n+1}^{\#}$.
\end{proof}

\begin{thm}\label{thm -1}
We have
\[
\psi_{-1}^{\#}(x_{1},\dots,x_{d})=\frac{1}{x_{1}+\cdots+x_{d}}\log(\paj)(x_{1},\dots,x_{d}).
\]
\end{thm}
\begin{proof}
We have
\begin{align*}
\psi_{-1}^{\#}&(x_{1},\dots,x_{d})  =\psi_{-1}(x_{1},x_{1}+x_{2},\dots,x_{1}+\cdots+x_{d})\\
 & =\frac{1}{x_{1}+\cdots+x_{d}}\sum_{h=1}^d\frac{(-1)^{h+1}}{h}\sum_{0=i_{0}<i_{1}<\cdots<i_{h}=d} \\
 & \qquad\quad \prod_{s=0}^{h-1}\frac{1}{x_{i_{s}+1}(x_{i_{s}+1}+x_{i_{s}+2})\cdots(x_{i_{s}+1}+x_{i_{s}+2}+\cdots+x_{i_{s+1}})}\\
 & =\frac{1}{x_{1}+\cdots+x_{d}}\sum_{h=1}^d\frac{(-1)^{h+1}}{h}\sum_{0=i_{0}<i_{1}<\cdots<i_{h}=d}\prod_{s=0}^{h-1}\paj(x_{i_{s}+1},x_{i_{s}+2},\dots,x_{i_{s+1}})\\
 & =\frac{1}{x_{1}+\cdots+x_{d}}\sum_{h=1}^d\frac{(-1)^{h+1}}{h}(\paj-I^{\bullet})^{h}(x_{1},\dots,x_{d})\\
 & =\frac{1}{x_{1}+\cdots+x_{d}}\log(\paj)(x_{1},\dots,x_{d}).
\end{align*}
\end{proof}

%%%%%%%%%%%%%%%%%%%%%%%%%%%%%%%%%%%%%%%%%%%%%%%%%%%%%%%%%%%%%%%%%%%%%%
\section{On polynomial solutions}\label{sec: On polynomial solutions}%power series solutions?
We compare, up to depth three,
the polynomial solutions $\sigma_{2n+1}^c$ ($n \geq 1$) constructed by Brown \cite{B17b}
and the polynomial solution $\luma_{2n+1}^{(3)}$ constructed by Ecalle  \cite{E-flex}
to the double shuffle equations modulo products.
In particular, the discrepancy between these two families in depth three is made explicit in terms of Bernoulli numbers and certain accompanying polynomials in Theorem \ref{thm: comparison depth 3}.

\subsection{Brown's polynomial solution}
This subsection recalls the construction of Brown's polynomial solution.

Let  $s'\in\ARI(\mathcal{F}_\Lau)/\ARI_{\geq3}(\mathcal{F}_\Lau)$ defined by
\footnote{
In the published version of \cite{B17b}, the coefficient in $s'(u_{1},u_{2})$
is $\frac{1}{6}$, and in the arXiv version, the coefficients is $\frac{1}{12}$.
It does not look that the construction of $\xi$  works well if the
coefficient is $\frac{1}{6}$.
}

\[
s'(u_{1})=\frac{1}{2x_{1}}, \quad
s'(u_{1},u_{2})=\frac{1}{12}\left(\frac{1}{x_{1}x_{2}}+\frac{1}{x_{2}(x_{1}-x_{2})}\right)
\]
where we put $x_{1}=u_{1}$, $x_{2}=u_{1}+u_{2}$.
For
$S\in\BIMU_1$, define
$$\xi'(S)=S + \ari(S,s')+
\frac{1}{2}\ari(\ari(S,s'),s')\in
\ARI(\mathcal{F}_\Lau)/\ARI_{\geq 4}(\mathcal{F}_\Lau).
$$
Brown (\cite[Definition 5.1]{B17b}) introduces the element
$$
\xi_{2n+1}:=\xi'(\sa_{2n+1})\in \ARI(\mathcal{F}_\Lau)/\ARI_{\geq 4}(\mathcal{F}_\Lau)
$$
and shows that it satisfies
the double shuffle relations modulo products in depth 2 and 3 (cf. \cite[Proposition 5.2]{B17b}).

In \cite[\S 5.2]{B17b},
he constructs an element
$$\sigma_{2n+1}^c\in\ARI(\mathcal{F}_\pol)/\ARI_{\geq 4}(\mathcal{F}_\pol),$$
called the {\it canonical normalization},
which satisfies  the double shuffle relations modulo products up to depth 3.
Up to length (depth) 3,
it is given by
\begin{align}\label{eq: sigma c 2n+1}
\sigma_{2n+1}^c
&\equiv
\xi_{2n+1}+
%\slang_1(\sa_{2n+1})+
\sum_{a+b=n}\frac{1}{24b}\frac{B_{2a}B_{2b}}{B_{2n}}\binom{2n}{2a}
\ari ({\sa_{2a+1}},\ari (\sa_{2b+1},\sa_{-1}))
\\ \notag
&\qquad\qquad\qquad\qquad\qquad\qquad\qquad\qquad\qquad\qquad
\bmod \ARI_{\geq 4}(\mathcal{F}_\Lau)
\end{align}
Its length $i$-part $(\sigma_{2n+1}^c)^{(i)}$  ($i=1,2,3$) is in
$\mathcal{F}_{\pol,i}$
%$\ARI(\mathcal{F}_\pol)/\ARI_{\geq 4}(\mathcal{F}_\pol)$
and is given by
\begin{itemize}
\item
$(\sigma_{2n+1}^c)^{(1)}=\xi_{2n+1}^{(1)}=\sa_{2n+1}$,
\item
$(\sigma_{2n+1}^c)^{(2)}=\xi_{2n+1}^{(2)}$,
\item
$(\sigma_{2n+1}^c)^{(3)}=
%\slang_1(\sa_{2n+1})^{(3)}+
\xi_{2n+1}^{(3)}+
\sum_{a+b=n}\frac{1}{24b}\frac{B_{2a}B_{2b}}{B_{2n}}\binom{2n}{2a}
\ari ({\sa_{2a+1}},\ari (\sa_{2b+1},\sa_{-1}))^{(3)}.
$
%that is,
%\begin{equation}\label{eq: sigma3}
%(\sigma_{2n+1}^c)^{(3)}=
%\xi_{2n+1}^{(3)}+
%\sum_{a+b=n}\frac{1}{24b}\frac{B_{2a}B_{2b}}{B_{2n}}\binom{2n}{2a}
%\ari ({u_1^{2a}},\ari (u_1^{2b},u_1^{-2})).
%\end{equation}
\end{itemize}

\subsection{Ecalle's polynomial solution}
This subsection recalls the construction of Ecalle's polynomial solution $\luma_{2n+1}^{(3)}$.
%Let
%$\Su_{[1,2]}^\bullet=\sum_m\Su^{m}_{[1,2]}\in \BIMU_2$
% given by
%$$\Su^{2n+1}_{[1,2]}=-\frac{1}{12}\sum_{a+b=n}
%\frac{B_{2a}B_{2b}}{B_{2n}}\binom{2n}{2a} u_1^{2a}u_2^{2b-1}
%$$
%and $\Su^{2n}_{[1,2]}=0$.
In \cite[\S 6.3 and \S 6.7]{E-flex}, Ecalle sketches a construction of
%constructed an
the element
$$\luma_{2n+1} \in\ARI(\mathcal{F}_\pol)$$ %/\ARI_{\geq 4}(\mathcal{F}_\ser)$
which satisfies  the double shuffle relations modulo products.
Up to length (depth) 3,
it is given by
\begin{align}\label{eq: luma 2n+1}
\luma_{2n+1}& \equiv \slang_1(\sa_{2n+1})%+\slang_{[1,2]}(\Su_{[1,2]})
%\\ \notag&
-\frac{1}{12}\sum_{a+b=n}
\frac{B_{2a}B_{2b}}{B_{2n}}\binom{2n}{2a}\ari(\slang_1(\sa_{2a+1}),\slang_2(\sa_{2b}))
%\in\ARI(\mathcal{F}_\ser)/\ARI_{\geq 4}(\mathcal{F}_\ser)
\\ \notag
&\qquad\qquad\qquad\qquad\qquad\qquad\qquad\qquad\qquad\qquad\qquad\qquad
\bmod \ARI_{\geq 4}(\mathcal{F}_\Lau).
\end{align}
%where $\Su_{[1]}^\bullet=\Sa_{[1]}^\bullet=\sum_m\sa^{2m+1}_{[1]}:=\frac{u^2}{1-u^2}\in\BIMU_1$.
%$$\Su^{2n+1}_{[1]}=u^{2n}, \quad  \Su^{2n}_{[1]}=0$$
Its length $i$-part  $\luma^{(i)}_{2n+1}$ ($i=1,2,3$) is in
$\mathcal{F}_{\pol,i}$
and is given by
\begin{itemize}
\item  $\luma^{(1)}_{2n+1}=\slang_1(\sa_{2n+1})^{(1)}=(\sigma_{2n+1}^c)^{(1)}=\xi_{2n+1}^{(1)}$,
\item $\luma^{(2)}_{2n+1}=\slang_1(\sa_{2n+1})^{(2)}=(\sigma_{2n+1}^c)^{(2)}=\xi_{2n+1}^{(2)}$,
\item $\luma^{(3)}_{2n+1}
%=\slang_1(\sa_{2n+1})^{(3)}+\slang_{[1,2]}(\Su^{2n+1}_{[1,2]})
=\slang_1(\sa_{2n+1})^{(3)}
-\frac{1}{12}\sum_{a+b=n}
\frac{B_{2a}B_{2b}}{B_{2n}}\binom{2n}{2a}\ari(\slang_1(\sa_{2a+1}),\slang_2(\sa_{2b}))^{(3)}
$.
%, that is,
%\begin{equation}\label{eq luma3}
%\luma^{(3)}_{2n+1}
%=%\xi_{2n+1}^{(3)}
%\slang_1(\sa_{2n+1})^{(3)}
%-\frac{1}{12}\sum_{a+b=n}
%\frac{B_{2a}B_{2b}}{B_{2n}}\binom{2n}{2a}\ari(\slang_1(u_1^{2a}),\slang_1(u_1^{2b-1}))^{(3)}.
%\end{equation}
\end{itemize}

\subsection{Comparison}
%\Add{Explain that both Ecalle and Brown do the same  desingularization method (?).}
In this subsection, a detailed comparison between Brown's element $\sigma_{2n+1}^c$ and Ecalle's element $\luma_{2n+1}$ is carried out, and their discrepancy in depth three is explicitly described in terms of Bernoulli numbers and certain polynomials.

\begin{thm}\label{thm: comparison depth 3}
(1). The first term of $\sigma_{2n+1}^c$ in \eqref{eq: sigma c 2n+1}
and the one of $\luma_{2n+1}$ in \eqref{eq: luma 2n+1}
coincides modulo $ARI_{\geq 4}(\mathcal{F}_\Lau)$, that is,
\[
\xi_{2n+1}\equiv \slang_{1}(\sa_{2n+1}) \ \bmod \ARI_{\geq 4}(\mathcal{F}_\Lau).
\]
Thus for $n\geq 1$, we have
$$
\sigma_{2n+1}^c-\luma_{2n+1}=\sum_{a+b=n}\
\frac{B_{2a}\cdot B_{2b}}{24 b \cdot B_{2n}}\binom{2n}{2a} D_{a,b}
\ \bmod \ARI_{\geq 4}(\mathcal{F}_\Lau)
$$
where
$$
D_{a,b}:=\ari \bigl({\sa_{2a+1}},\ari (\sa_{2b+1},\sa_{-1})\bigr)+2b\ari\bigl(\slang_1(\sa_{2a+1}),\slang_2(\sa_{2b})\bigr)
\in\ARI_{\geq 3}(\mathcal{F}_\Lau). %/\ARI_{\geq 4}(\mathcal{F}_\Lau).
$$

(2).  The length (depth)-three component of each summand
$D_{a,b}^{(3)}$ belongs to $\mathcal{F}_{\pol,3}=\Q[x_1,x_2,x_3]$.
 %$\mathcal F_{\pol,3}$.
\end{thm}

%\Add{The proof of the theorem was corrected (2025.12.10).}

\begin{proof}

The first claim is equivalent to the identity
\[
\xi'(S)\equiv\slang_{1}(S) \ \bmod \ARI_{\geq 4}(\mathcal{F}_\Lau)
\]
for $S\in\mathcal{M}_{1}(\mathcal{F}_\pol)$ with $S(x)=S(-x)$.
This identity can be verified by a direct computer calculation.
See \url{https://github.com/MinoruHirose/MouldTheory} for the code used in this verification.

Next, let us show the second claim. Since
\[
\ari(\sa_{2b+1},\sa_{-1})=\arit(\sa_{-1})(\sa_{2b+1})-\arit(\sa_{2b+1})(\sa_{-1})+\sa_{2b+1}\times\sa_{-1}-\sa_{-1}\times\sa_{2b+1},
\]
we have
\begin{align*}
\left(\ari(\sa_{2b+1},\sa_{-1})\right)^{2}(x_{1},x_{2}) & =(x_{1}+x_{2})^{2b}\left(\frac{1}{x_{1}^{2}}-\frac{1}{x_{2}^{2}}\right)-\frac{1}{(x_{1}+x_{2})^{2}}\left(x_{1}^{2b}-x_{2}^{2b}\right)+\frac{x_{1}^{2b}}{x_{2}^{2}}-\frac{x_{2}^{2b}}{x_{1}^{2}}\\
 & =\frac{(x_{1}+x_{2})^{2b}-x_{2}^{2b}}{x_{1}^{2}}-\frac{(x_{1}+x_{2})^{2b}-x_{1}^{2b}}{x_{2}^{2}}-\frac{x_{1}^{2b}-x_{2}^{2b}}{(x_{1}+x_{2})^{2}}\\
 & \in2b\frac{x_{2}^{2b-1}}{x_{1}}-2b\frac{x_{1}^{2b-1}}{x_{2}}+2b\frac{x_{2}^{2b-1}}{x_{1}+x_{2}}+\mathbb{Q}[x_{1},x_{2}]
\end{align*}
On the other hand,
\begin{align*}
\left(\slang_{2}(\sa_{2b})\right)^{2}(x_{1},x_{2}) & =\frac{x_{1}-x_{2}}{2x_{1}x_{2}}(x_{1}+x_{2})^{2b-1}+\frac{x_{1}+2x_{2}}{2x_{2}(x_{1}+x_{2})}x_{1}^{2b-1}-\frac{2x_{1}+x_{2}}{2x_{1}(x_{1}+x_{2})}x_{2}^{2b-1}\\
 & =\frac{(x_{1}+x_{2})^{2b-1}+x_{1}^{2b-1}}{2x_{2}}-\frac{(x_{1}+x_{2})^{2b-1}+x_{2}^{2b-1}}{2x_{1}}+\frac{x_{1}^{2b-1}-x_{2}^{2b-1}}{2(x_{1}+x_{2})}\\
 & \in -\frac{x_{2}^{2b-1}}{x_{1}}+\frac{x_{1}^{2b-1}}{x_{2}}-\frac{x_{2}^{2b-1}}{x_{1}+x_{2}}+\mathbb{Q}[x_{1},x_{2}].
\end{align*}
Thus,
\[
\ari(\sa_{2b+1},\sa_{-1})+2b\slang_{2}(\sa_{2b})\in\ARI(\mathcal{F}_{\pol}) + \ARI_{\geq3}(\mathcal{F}_{\Lau}).
\]
Therefore
\begin{align*}
D_{a,b} & =\ari\bigl(\sa_{2a+1},\ari(\sa_{2b+1},\sa_{-1})+2b\slang_{2}(\sa_{2b}\bigr)\\
 & \quad+\ari\bigl(\slang_{1}(\sa_{2a+1})-\sa_{2a+1},2b\slang_{2}(\sa_{2b})\bigr)\\
 & \in\ARI(\mathcal{F}_{\pol})+\ARI_{\geq4}(\mathcal{F}_{\Lau})
\end{align*}
% \[
% \ari(\sa_{2a+1},\ari(\sa_{2b+1},\sa_{-1}))+2b\ari(\slang_{1}(\sa_{2a+1}),\slang_{2}(\sa_{2b}))\in\ARI(\mathcal{F}_{\pol}) + \ARI_{\geq4}(\mathcal{F}_{\ser})
% \]
since
\[
\ari(\sa_{2a+1}, \ARI(\mathcal{F}_{\pol})) \subset \ARI(\mathcal{F}_{\pol}),
\]
\[
\ari(\sa_{2a+1}, \ARI_{\geq3}(\mathcal{F}_{\Lau})) \subset \ARI_{\geq4}(\mathcal{F}_{\Lau})
\]
\[
\slang_{1}(\sa_{2a+1})-\sa_{2a+1},\ \slang_{2}(\sa_{2b}) \in \ARI_{\geq2}(\mathcal{F}_{\Lau}),
\]
and
\[
\ari\bigl(\ARI_{\geq2}(\mathcal{F}_{\Lau}), \ARI_{\geq2}(\mathcal{F}_{\Lau})\bigr) \subset \ARI_{\geq4}(\mathcal{F}_{\Lau}).
\]
Hence the theorem is proved.
\end{proof}

%\Add{Can one show  $(\sigma_{2n+1}^c)^{(3)} \neq \luma^{(3)}_{2n+1}$?}

{\it Acknowledgments.}
H.F. has been supported by grants JSPS KAKENHI JP24K00520 and JP24K21510.
M.H. has been supported by grants JSPS KAKENHI JP22K03244.
N.K. has been supported by grants JSPS KAKENHI JP23KJ1420.

%%%%%%%%%%%%%%%%%%%%%%%%%%%%%%%%%%%%%%%%%%%%%%%%%%%%%%%%%%%%%%%%%%%%%%

\end{document}